\documentclass[11pt]{article}
\usepackage[latin1]{inputenc}
\usepackage[margin=0.95in]{geometry} %default geometry package already reduces margins
\usepackage{amsmath}
\usepackage{amssymb}
\usepackage{amsfonts}
\usepackage{amsthm}
\usepackage{url}

\usepackage[dvipsnames]{xcolor}

\usepackage{tikz}
\usetikzlibrary{trees,arrows}
\usepackage{subfigure}
\usepackage{caption}
\newcommand{\E}{\ensuremath{\mathbb{E}}}
\newcommand{\V}{\ensuremath{\mathrm{Var}}}

\newcommand{\R}{\ensuremath{\mathbb{R}}}

\newtheorem{thm}{Theorem}[section]

\newtheorem{lem}[thm]{Lemma}
\newtheorem{rem}[thm]{Remark}

\begin{document}
\title{A  note on the independence number, domination number and related parameters of random binary search trees and random recursive trees}
\date{\today}
\author{Michael Fuchs\\Department of Mathematical Sciences\\
National Chengchi University\\
Taiwan \and
Cecilia Holmgren\footnote{The work of Cecilia Holmgren is supported by the Swedish Research Council, the Ragnar S\"oderberg Foundation and the Knut and Alice  Wallenberg Foundation}\\ Department of Mathematics\\ Uppsala University \\ Sweden \and Dieter Mitsche\footnote{Dieter Mitsche has been supported by IDEXLYON of Universit\'{e} de Lyon (Programme Investissements d'Avenir ANR16-IDEX-0005)}\\ Institut Camille Jordan (UMR 5208)\\ Univ.~de Lyon, Univ. Jean Monnet\\ France \and  Ralph Neininger\\ Institute for Mathematics\\
Goethe University Frankfurt\\ Germany}

\maketitle

\abstract{We identify the mean growth of the independence number of random binary search trees and random recursive trees and show  normal fluctuations around their means. Similarly we also show normal limit laws for the domination number and variations of it for these two cases of random tree models. Our results are an application of a recent general theorem of Holmgren and Janson on fringe trees in these two random tree models.
}
\section{Introduction and results}\label{sec:1}
In this note we study the independence number, the domination number and related parameters of random binary search trees and random recursive trees asymptotically. First, in Section~\ref{indsection}, we derive asymptotics for the mean and variance and provide central limit laws for the independence number of both tree models. This covers a few other graph parameters which are affine functions of the independence number, see Remark \ref{rem_thm}(c) below. In Section \ref{dom}, we also provide central limit laws for the domination number and related parameters for both of these cases of random tree models. Finally, albeit coinciding with the independence number on trees, we also give a direct proof of such a theorem for the clique cover number in Section~\ref{ccsection}.

We first recall the parameters under consideration and present the models of trees we are looking at and state our results.\\

\noindent
{\bf Independence and domination number.} The {\em independence number} of a graph is the size of a maximum independent set in the graph, where an independent set is a subset of the vertices of the graph so that no two vertices of this subset are connected (are neighbors) within the graph. The independence number is an important and well-known graph parameter: besides its applications in scheduling theory, coding theory and collusion detection in voting pools (see~\cite{Coding, Scheduling, Collusion}), it has attracted a lot of interest especially in theoretical computer science: the independence number is well known to be NP-hard to compute in general, see for example~\cite{Robson}. Since then, exact fast exponential algorithms have been developed (see~\cite{Exact1, Exact2}) as well as polynomial-time algorithms for special graph classes (claw-free graphs, $P_5$-free graphs, perfect graphs, see~\cite{Poly1, Poly2, Poly3}). In general, it is also NP-hard to approximate the independence number (that is, it is not possible to approximate it up to a constant factor in polynomial time)~\cite{Bazgan}, but again for special graph classes such as planar graphs, or more generally, for graphs closed under taking minors, polynomial-time approximation schemes do exist~\cite{Baker, Grohe}. For bipartite graphs, thus in particular trees, by K\"{o}nig's theorem, all vertices  not in the minimum vertex cover can be included in a maximum independent set (see also the remark below), and thus the independence number can be found in polynomial time.
In combinatorics, it has also received considerable attraction, starting with the early work by Bollob\'{a}s~\cite{Boll_81}.

Given a finite graph $G$ with vertex set $V$, a subset $W\subset V$ is called a dominating set for $V$ if every vertex in $V$ lies at graph distance at most 1 from $W$. The {\em domination number}  of $G$ is then defined to be the minimum number $m$ such that there exists a dominating set $W$ of size $m$. Finding dominating sets is important in finding `central' or `important' sets of vertices in a network, in contexts such as facility location~\cite{HaHeSl1_1998}, molecular biology~\cite{MMBP} and in wireless networks~\cite{SSZ}.
Dominating sets have attracted considerable attention in
discrete mathematics (see~\cite{HaHeSl1_1998,HaHeSl2_1998} and~\cite{HeLa_1990}) and as in the case of the independence number, in theoretical computer science: it was shown already in the 1970s (see~\cite{Karp}) that the domination number is NP-hard to compute, and it is also NP-hard to approximate up to a logarithmic factor in general~\cite{Raz}. Since then, as in the case of the independence number, exact fast exponential algorithms have been developed~\cite{Fomin, vanRooij}, and faster algorithms for special graph classes have been found as well (see for example~\cite{Seriesparallel} for series-parallel graphs). For trees, linear-time algorithms are known~\cite{Cockayne}.\\

\noindent
{\bf Random recursive tree and random binary search tree.}
A {\em random recursive tree} is a labelled rooted tree which can be constructed as follows. For the first step we start with the root vertex labelled $1$. In the $n$-th step, $n\ge 2$, one of the existing vertices labelled $1,\ldots,n-1$ is chosen uniformly at random where a vertex with label $n$ is attached. Subsequently, a random recursive tree with $n$ vertices is denoted by $\Lambda_n$. For reference see the survey of Smythe and
Mahmoud~\cite{smma94}. We will need the following fact: A random recursive tree with $n$ vertices can be cut into two trees by removing the edge between the root vertex labelled $1$ and the vertex labelled $2$. This yields two trees both with a random size, both sizes being uniformly distributed on $\{1,\ldots,n-1\}$. Moreover, conditional on their sizes, these two trees are independent and both are (after proper relabelling of their vertices) random recursive trees of their respective size.

The {\em  random binary search tree} can be constructed from a uniformly distributed random permutation $(\Pi_1,\ldots,\Pi_n)$ of $\{1,\ldots,n\}$.  The first number $\Pi_1$ becomes the
root of the
tree. Then the numbers $\Pi_2,\ldots,\Pi_n$ are successively inserted recursively. Each number
is compared with the root. If it is smaller than the root, it is directed
to the root's left subtree, otherwise to its right subtree. There, this procedure
is recursively iterated until an empty subtree is reached, where the
number is inserted as a new vertex. Subsequently, a random binary search tree with $n$ vertices is denoted by ${\cal T}_n$. For reference see Knuth~\cite{kn98}. We need the following decomposition property: The left and right subtrees at the root of the binary search tree both have random sizes uniformly distributed on $\{0,\ldots,n-1\}$. Conditional on their sizes they are independent and both are (after proper relabelling of their vertices) random binary search trees of their respective sizes.\\

\noindent {\bf Results on the independence number.} We denote by $I_n$ the independence number of ${\cal T}_n$ and by $\widehat{I}_n$ the independence number of $\Lambda_n$. We have the following asymptotic results:
\begin{thm}\label{thm_bst} For the independence number $I_n$ of a random binary search tree with $n$ vertices we have, as $n\to\infty$, that
$\E[I_n]= \mu n +\mathrm{O}(1)$, $\V(I_n)\sim \sigma^2 n$ and
\begin{align*}
\frac{I_n-\mu n}{\sqrt{n}}\stackrel{d}{\longrightarrow} {\cal N}(0,\sigma^2)
\end{align*}
with
\begin{align}\label{int_rep_mu}
\mu=2(\sqrt{5}-3)\int_{0}^{1}\frac{x^{\sqrt{5}}-1}{(3\sqrt{5}-7)
x^{\sqrt{5}}+2}{\rm d}x=0.54287631\ldots
\end{align}
and  a constant $\sigma>0$.
\end{thm}

\begin{thm}\label{thm_rrt} For the independence number $\widehat{I}_n$ of a random recursive tree with $n$ vertices we have, as $n\to\infty$, that $\E[\widehat{I}_n]= \widehat{\mu} n +\mathrm{O}(1)$, $\V(\widehat{I}_n)\sim \widehat{\sigma}^2 n$ and
\begin{align*}
\frac{\widehat{I}_n-\widehat{\mu} n}{\sqrt{n}}\stackrel{d}{\longrightarrow} {\cal N}(0,\widehat{\sigma}^2)
\end{align*}
with the Euler--Gompertz constant
\begin{align}\label{euler_rep}
\widehat{\mu}=\int_0^1\frac{1}{1-\log x}{\rm d}x=0.59634736\ldots
\end{align}
and a constant $\widehat{\sigma}>0$.
\end{thm}
\begin{rem}\label{rem_wagner}{\rm  Stephan Wagner (Stellenbosch University) informed us that he and his student Kenneth Dadedzi have an independent approach to results similar of our Theorems \ref{thm_bst} and \ref{thm_rrt}; they use generating functions to determine the spectrum of the Laplacian operator on these trees, see \cite{dadedzi}. Stephan also informed us that our representation (\ref{def_mu_rrt}) for $\widehat{\mu}$ has the explicit integral representation given in (\ref{euler_rep}).
}
\end{rem}

\noindent
\textbf{Results on the domination number.} For the domination number of random binary search trees and random recursive trees we have similar results.
\begin{thm}\label{thm_bstrstdom} For the domination number $D_n$ of a random binary search tree with $n$ vertices we have, as $n\to\infty$, that
$\E[D_n]= \nu n +\mathrm{O}(1)$, $\V(D_n)\sim \tau^2 n$ with some constants $\nu, \tau>0$  and
\begin{align*}
\frac{D_n-\nu n}{\sqrt{n}}\stackrel{d}{\longrightarrow} {\cal N}(0,\tau^2).
\end{align*}
Similarly, for the domination number $\widehat{D}_n$ of a random recursive tree with $n$ vertices we have, as $n\to\infty$, that
$\E[\widehat{D}_n]= \widehat{\nu} n +\mathrm{O}(1)$, $\V(D_n)\sim \widehat{\tau}^2 n$ with some constants $\widehat{\nu},\widehat{\tau}>0$ and
\begin{align*}
\frac{\widehat{D}_n-\widehat{\nu} n}{\sqrt{n}}\stackrel{d}{\longrightarrow} {\cal N}(0,\widehat{\tau}^2).
\end{align*}
\end{thm}
\begin{rem}{\rm
A variation of the domination number, the so-called {\em $k$-domination number} of a graph, was introduced in~\cite{FJ_84}. This is defined as the minimum size of a set $S$ of vertices in a graph such that each vertex of the graph (outside the set $S$) has at least $k$ neighbors in $S$. We can analyze these numbers as well in the case of random binary search trees and random recursive trees and obtain normal limit laws corresponding to the ones in Theorem \ref{thm_bstrstdom}. However for binary search trees, where each vertex has degree at most 3, we also have to assume that $k\leq 3$ (to avoid the trivial case $|S| =n$), while for random recursive trees, we may consider the $k$-domination number for any constant $k>0$.
}
\end{rem}

\noindent
\begin{rem}\label{rem_thm}
{\rm
(a) Various quantities for random binary search trees have systematically been studied with respect to limit distributions by Devroye~\cite{dev_03} and Hwang and Neininger~\cite{hwne_02}. However, the independence number and the domination number do not fit under the assumptions made in those two studies. Our proof relies on a recent refined study of fringe trees of random binary search trees and random recursive trees of Holmgren and Janson~\cite{HJ15} which extends parts of the results of~\cite{dev_03,hwne_02}.

(b) Holmgren and Janson~\cite{HJ15} also give a general formula for variances which covers our variances $\sigma^2$ and $\widehat{\sigma}^2$ in Theorems~\ref{thm_bst} and~\ref{thm_rrt}. Their representation, in principle, allows to also give numerical approximations for $\sigma^2$ and $\widehat{\sigma}^2$.

(c) There are a few (other) related graph parameters which are covered by our results, since they are affine functions of the independence number:
The {\em matching number} (also known as edge independence number)  is the size of a maximum set of edges so that no two edges have a common vertex.
For all bipartite graphs and in particular trees, the matching number and the independence number add up to the size of the tree. Hence, for the matching numbers $M_n$ and $\widehat{M}_n$ of a random binary search tree and a random recursive tree with $n$ vertices respectively, we have $\E[M_n]=(1-\mu)n+\mathrm{O}(1)$ with the same variance and limit as for $I_n$ in Theorem \ref{thm_bst}, and $\E[\widehat{M}_n]=(1-\widehat{\mu})n+\mathrm{O}(1)$ with the same variance and limit as for $\widehat{I}_n$ in Theorem \ref{thm_rrt}.

The {\em edge cover number} of a connected graph is the minimum number of edges so that all vertices are incident to at least one edge. The edge cover number and the independence number coincide for trees.

The {\em vertex cover number} is the minimum number  of vertices such that every edge has at least one of these vertices as an endpoint. The matching number and the vertex cover number coincide for trees.

The multiplicity of the eigenvalue $1$ of the {\em normalized Laplacian operator} of a tree is twice the independence number of the tree minus its size, see \cite[Theorem 1]{chjo}. Hence, Theorems \ref{thm_bst} and \ref{thm_rrt} imply the asymptotics of this multiplicity of the two random tree models considered in the present note as well. See \cite{dadedzi} for a more general study of the asymptotics of the spectra of these random trees.

The {\em clique cover number} of a finite graph $G$ is the minimum number of colors needed to color properly the vertices of the complement of $G$ (the complement of $G$ has the same vertex set as $G$, and two vertices are adjacent in the complement of $G$ if and only if they are not adjacent in $G$). For trees, the clique cover number coincides with the independence number. We give a variant of the derivation of Theorems \ref{thm_bst} and \ref{thm_rrt} in terms of the clique cover number, see
section  \ref{ccsection}.
}
\end{rem}

\section{Independence number}\label{indsection}
For our proof we use a simple construction of a maximum independent set by starting at the leaves. For a rooted tree $T$ (or a forest of rooted trees) denote by $\mathrm{leaf}(T)$ the set of leaves of $T$ and by $p(\mathrm{leaf}(T))$ the set of the parents of the leaves of $T$. Recursively, define
$$T^{[0]}:=T \quad \mbox{and}\quad   T^{[\ell]}:=T^{[\ell-1]}\setminus \big(\mathrm{leaf}(T^{[\ell-1]})\cup p(\mathrm{leaf}(T^{[\ell-1]})\big) \;  \mbox{for}\; \ell\ge 1.$$
So, $T^{[0]},T^{[1]},T^{[2]},\ldots$ is a sequence of rooted trees or forests of rooted trees starting with $T$ where in each step all the leaves together with their parents are removed from the present tree or forest until we reach the empty graph. Note, that when starting with a tree the sequence generated may also contain forests.
% (By slight abuse of notation, by $T\setminus W$, where $W$ is a subset of the vertices of $T$ we denote that we remove the vertices $W$ and the edges incident to $W$ from $T$.) \marginpar{\tiny Dieter: Or is this clear and common practice in graph theory?}
\begin{lem}\label{lem_struc}
Let $T$ be a rooted tree. Then
$$\bigcup_{\ell=0}^\infty \mathrm{leaf}\left(T^{[\ell]}\right)$$ is a maximum independent set of $T$.
\end{lem}
\begin{proof}
Let $T$ be a rooted tree  or forest of rooted trees. We first show that there is always a maximum independent set of $T$ which contains $\mathrm{leaf}(T)$. To see this choose an arbitrary maximum independent set $A$ of $T$. If $A$ does not contain a leaf $\nu$ then it has to contain its parent $p(\nu)$. However, then also $(A\setminus \{p(\nu)\})\cup \{\nu\}$ is a maximum independent set of $T$ which now contains the leaf $\nu$. Iterating this process implies the existence of a maximum independent set of $T$ containing $\mathrm{leaf}(T)$.

Further, a maximum independent set containing $\mathrm{leaf}(T)$ cannot contain any vertex of $p(\mathrm{leaf}(T))$ and hence consists of the union of $\mathrm{leaf}(T^{[0]})$ and a maximum independent set of $T^{[1]}$. Applying the previous argument to $T^{[1]}$ and using induction implies the assertion. \end{proof}

Subsequently, we call the maximum independent set of a rooted tree constructed in Lemma \ref{lem_struc} the {\em layered independent set}. \\

%Similar to Lemma \ref{lem_struc}, we find that a minimal dominating set of the tree is given by $\cup_{\ell=0}^\infty \mathrm{leaf}\left(T^{[2\ell+1]}\right)$.
%Hence, independence number and the dominance number sum up to the size of the tree as mentioned in Remark \ref{rem_thm} (c).

\noindent
{\bf A result of Holmgren and Janson~\cite{HJ15}.} Recalling notions from Holmgren and Janson~\cite{HJ15} a  {\em functional} of trees is a real-valued function of trees. For a rooted tree $T$ and a vertex $v\in T$ the {\em fringe tree} $T(v)$ is the subtree rooted at $v\in T$ which consists of all descendants of $v$ in $T$. For a functional $f$ of rooted trees we define
\begin{align}\label{ref_F}
F(T)=F(T;f):=\sum_{v\in T} f(T(v)).
\end{align}
Corollary 1.15 in~\cite{HJ15} states that for a functional $f$ with the growth condition $f(T)=\mathrm{O}(|T|^\alpha)$ for some $\alpha<\frac{1}{2}$ and the random binary search tree ${\cal T}_n$  we have $\E[F({\cal T}_n)]\sim \mu_F n$, $\V(F({\cal T}_n))\sim \sigma_F^2 n$ as $n\to\infty$, and that $F({\cal T}_n)$, after normalization, is asymptotically normal distributed. The constant $\mu_F$ is given by
\begin{align}\label{def_muF}
\mu_F=\sum_{k=1}^\infty \frac{2\,\E[f({\cal T}_k)]}{(k+1)(k+2)}.
\end{align}
Note that in (1.25) in~\cite{HJ15} also an expression for $\sigma_F^2$ is given. A similar result also holds for the random recursive tree $\Lambda_n$, where the corresponding constant $\widehat{\mu}_F$ is given by
\begin{align}\label{def_muFhat}
\widehat{\mu}_F=\sum_{k=1}^\infty \frac{\E[f(\Lambda_k)]}{k(k+1)}.
\end{align}
Further note that the proofs in~\cite{HJ15} also imply that $\E[F({\cal T}_n)]= \mu_F n +\mathrm{O}(1)$ and that $\E[F(\Lambda_n)]= \widehat{\mu}_F n +\mathrm{O}(1)$ under the stronger growth assumption that $f(T)=\mathrm{O}(1)$.

Putting things together now implies Theorems~\ref{thm_bst} and~\ref{thm_rrt}:\\

\noindent
{\em Proof of Theorem~\ref{thm_bst} and Theorem~\ref{thm_rrt}.}
Note that the independence number of a rooted tree can be covered as a function $F$ in (\ref{ref_F}) as follows. We set $f$ as the indicator function
\begin{align*}
f(T):=\left\{ \begin{array}{cl} 1, & \mbox{if the root of } T \mbox{ is contained in the layered independent set of } T,\\
0, & \mbox{otherwise.}\end{array}\right.
\end{align*}
The structure of the layered independent set in Lemma~\ref{lem_struc} implies that any vertex $v\in T$ is contained in the layered independent set of $T$ if and only if it is contained in the layered independent set of $T(v)$.

Hence, the independence number of $T$ is given by $F(T)=\sum_{v\in T} f(T(v))$ as in (\ref{ref_F}). This implies that $I_n =F({\cal T}_n)$ and  $\widehat{I}_n =F(\Lambda_n)$ in distribution. We have $f(T)=\mathrm{O}(1).$ Hence, Corollary 1.15 of Holmgren and Janson~\cite{HJ15} implies the assertions of Theorem~\ref{thm_bst} and Theorem~\ref{thm_rrt} where $\sigma,\widehat{\sigma}>0$ follows from numerical computation (see Remark~\ref{rem_thm}(b)) and only the constants $\mu$ and $\widehat{\mu}$ need to be identified. In view of (\ref{def_muF}) and (\ref{def_muFhat}) we need to find $\E[f({\cal T}_k)]$ and $\E[f(\Lambda_k)]$.

For the random recursive tree $T$  note that $T$ can be cut into two trees
by removing the edge between the root vertex labelled $1$ and the vertex labelled $2$. We denote the two resulting trees by $T_1$ and
 $T_2$. Now, the root of $T$ is contained in the layered independent set of $T$ if and only if the root of $T_1$ is contained in the layered independent set of $T_1$ and the root of $T_2$ is not contained in the  layered independent set of $T_2$.  Now, the decomposition property of the random recursive tree mentioned in the  introduction implies that with $\widehat{p}_n=\E[f(\Lambda_n)]$ we have the recurrence
 \begin{align}\label{def_pnhat}
\widehat{p}_n=\frac{1}{n-1}\sum_{j=1}^{n-1}(1-\widehat{p}_j)\widehat{p}_{n-j},\quad n\ge 2,
\end{align}
with initial condition $\widehat{p}_1:=1$. Furthermore, for the constant $\widehat{\mu}$ in Theorem \ref{thm_rrt} we have the representation
\begin{align}\label{def_mu_rrt}
\widehat{\mu}=\sum_{k=1}^\infty \frac{\widehat{p}_k}{k(k+1)}.
\end{align}
Now, to find the integral expression for $\widehat{\mu}$ in (\ref{euler_rep}) consider the generating function
\begin{align*}
\widehat{P}(z):=\sum_{k\ge 1}\widehat{p}_k z^k.
\end{align*}
From (\ref{def_pnhat}) and the initial conditions we obtain
\begin{align*}
z\widehat{P}'(z)=-\widehat{P}(z)^2+\frac{1}{1-z}\widehat{P}(z).
\end{align*}
This Riccati equation can be solved by standard methods: We define $\widehat{Q}(z)$ as $\widehat{P}(z)=z\widehat{Q}'(z)/\widehat{Q}(z)$ and obtain
\begin{align*}
\widehat{Q}''(z)=\frac{1}{1-z}\widehat{Q}'(z)
\end{align*}
which implies $\widehat{Q}'(z)=(1-z)^{-1}$ and thus $\widehat{Q}(z)=-\log(1-z)+c$ with a constant $c\in \R$. Hence, we obtain
\begin{align*}
\widehat{P}(z)=\frac{z}{(1-z)(-\log(1-z)+c)}
\end{align*}
and the initial condition $\widehat{P}'(0)=1$ yields $c=1$. Now we obtain
\begin{align*}
\widehat{\mu}=\sum_{k=1}^\infty \frac{\widehat{p}_k}{k(k+1)}
&= \int_0^1 \int_0^t \frac{1}{(1-z)(1-\log(1-z))}{\rm d}z{\rm d}t\\
&= \int_0^1 \int_z^1 \frac{1}{(1-z)(1-\log(1-z))}{\rm d}t{\rm d}z\\
&=\int_0^1 \frac{1}{(1-\log(1-z))}{\rm d}z,
\end{align*}
which, after substitution, is the expression in (\ref{euler_rep}) for the Euler--Gompertz constant. This concludes the proof of Theorem \ref{thm_rrt}.

For the binary search tree case note that the root of the tree $T$ is contained in its  layered independent set if and only if both children $v_\mathrm{\ell}$ and $v_\mathrm{r}$ of the root are not contained in the layered independent set of $T(v_\mathrm{\ell})$ and  $T(v_\mathrm{r})$ respectively. Now, the decomposition property of the random binary search tree mentioned in the introduction implies that with $p_k=\E[f({\cal T}_k)]$ we have the relation
\begin{align}\label{def_pn}
p_n:=\frac{1}{n}\sum_{j=0}^{n-1}(1-p_j)(1-p_{n-1-j}),\quad n\ge 1,
\end{align}
with initial value $p_0:=0$  and for $\mu$ in Theorem \ref{thm_bst} that
\begin{align}\label{def_mu}
\mu=\sum_{k=0}^\infty \frac{2p_k}{(k+1)(k+2)}.
\end{align}
Now, a similar derivation as for the previous case implies the integral representation for $\mu$ in (\ref{int_rep_mu}). \hfill $\square$\\

\section{Domination number}\label{dom}
In this section we will see that we again can apply Corollary 1.15 in~\cite{HJ15} (on normal limit laws for the number of fringe trees) to deduce normal limit laws for the domination number in the case of random binary search trees and random recursive trees. Note that the domination number is not directly related to the independence number; in particular it is not an affine function of the independence number.

{\em Proof of Theorem~\ref{thm_bstrstdom}.}
Let $T$ be a rooted tree with $n$ vertices and let $S$ be a minimum size dominating set of $T$ and let $D(T):=|S|$ be its size. In order to analyze the domination number we introduce the following descriptions of so-called {\em root-dependent and root-independent dominating sets.} Let  $r$ be the root vertex of $T$. We say that a subset of the vertices of $T\setminus r$ is a {\em root-dependent dominating set} if it is a minimum dominating set of $T\setminus r$ of size $D(T)-1$ (i.e., the dominating set becomes strictly smaller when the root is left out). If this is not possible, i.e., $D(T\setminus r)>D(T)-1$  we say that a minimum dominating set of the tree $T$ is a {\em root-independent dominating set}.

Let $T$ be a rooted tree, with $m$ children $i=1,2,\dots,m$, of the root $r$. Observe that the domination number $D(T)$ is bounded from above by $D(T_1)+\cdots +D(T_m)+1$ (we dominate each tree separately and then add the root), and from below by $D(T_1)+\cdots +D(T_m)-m+1$ (we add the root and manage to dominate each subtree minus its root by using a root-dependent dominating set).

It is now clear that we can construct a minimum dominating set $S$ of $T$ so that $v\in S$, if and only if, $v$ is contained in a root-independent dominating set of $T(v)$ except for maybe the root vertex $r$.

Indeed, if we have a minimum dominating set $S$ with a  vertex $v\neq r$ (not equal to the root vertex of $T$)  that is not contained in $S$ and $T(v)$ has a root-independent dominating set containing $v$, then all the vertices of $S$ from $T(v)$ form a dominating set of $T(v)\setminus v$. However, since $T(v)\setminus v$ has no root-dependent dominating set, this implies that this set is also a dominating set of $T(v)$. Thus, it can be replaced by the root-independent dominating set of $T(v)$ which contains $v$ without increasing the cardinality of $S$.

On the other hand, if $S$ contains a  vertex $v\neq r$ that is not contained in a root-independent dominating set of $T(v)$ it could either be because no root-independent dominating set exists or because every root-independent dominating set excludes $v$. We  now show that in both cases we can again modify $S$ without increasing the size.

In the first case, that is, if no such set exists, i.e., we have a root-dependent dominating set of $T(v)\setminus v$, we could remove $v$ from $S$ and replace it with its parent and then replace elements of $S$ coming from $T(v)\setminus v$ with a root-dependent dominating set of $T(v)\setminus v$. This does not increase the size of $S$ and it is still dominating.

In the second case, that is, if $T(v)$ has a root-independent dominating set, but no root-independent dominating set of $T(v)$ contains $v$, then we could replace $v$ with its parent, and use the root-independent dominating set on $T(v)$ instead (since no root-independent dominating set of $T(v)$ contained $v$, it must be inefficient to include $v$ in $S$ if we only wanted to dominate $T(v)$).

We can now do this construction inductively ending at the root $r$ (where we cannot "push up" our dominating set anymore). The root $r$ will be included if any child of the root has a corresponding subtree with a root-dependent dominating set or if all of these subtrees have root-independent dominating sets which all exclude their root (we call this the Property A).

Thus, the domination number of a rooted tree $T$ can be covered as a function $F$ in (\ref{ref_F}) as follows. For every vertex $v\neq r$ we set $f$ as the indicator function
\begin{align*}
f_\mathrm{dom}(T(v)):=\left\{ \begin{array}{cl} 1, & \mbox{if $v$} \mbox{ is contained in a root-independent dominating set of $T(v)$},\\
0, & \mbox{otherwise,}\end{array}\right.
\end{align*}
whereas for the root vertex $r$ with subtree $T(r)=T$ (the whole tree) we set $f$ as the indicator function
\begin{align*}
f_\mathrm{dom}(T(r)):=\left\{ \begin{array}{cl} 1, & \mbox{if $T(r)$ satisfies Property A},\\
0, & \mbox{otherwise.}\end{array}\right.
\end{align*}

Hence, the domination number of $T$ is given by $F_\mathrm{dom}(T)=\sum_{v\in T} f_\mathrm{dom}(T(v))$ as in (\ref{ref_F}). This implies that the domination numbers $D_n:=D({\cal T}_n) =F_\mathrm{dom}({\cal T}_n)$ and  $\widehat{D}_n:=D(\Lambda_n) =F_\mathrm{dom}(\Lambda_n)$ in distribution. We have $f_\mathrm{dom}(T(v))=\mathrm{O}(1).$ Hence, Corollary 1.15 of Holmgren and Janson~\cite{HJ15} implies the assertions of Theorem \ref{thm_bstrstdom}. \hfill$\square$

\section{Clique cover number}\label{ccsection}
Computing the clique cover number, see Remark \ref{rem_thm}(c), of a general graph is NP-hard~\cite{Karp}, and it is also NP-hard to approximate it up to a factor $n^{1-\varepsilon}$ for any $\varepsilon > 0$~\cite{Zuckerman}. However, it is well known that for triangle-free graphs, in particular trees, the clique cover number coincides with the independence number on trees, see~\cite{Goddard}.
%(and it can thus be computed in polynomial time using a maximum matching algorithm, see also the remark below).
Hence, the clique cover number of random binary search trees and random recursive trees is covered by Theorems \ref{thm_bst} and \ref{thm_rrt}. However, in this section we give a direct proof of Theorems \ref{thm_bst} and \ref{thm_rrt} for the clique cover number to show that this parameter can also be captured by the fringe tree representation and Corollary 1.15 in~\cite{HJ15}.

\par
{\em Proof:} For a tree $T$, consider $T(v)$ with root $v$ and subtrees $T_1, \ldots, T_k$ with corresponding roots $v_1, \ldots, v_k$ that are the children of $v$. For such a tree $T(v)$, let $\mathcal{E}_v$ be the indicator event that there exists a subtree $T_i$ and an optimal clique coloring of the vertices of $T_i$ (that is, a coloring using a minimal number $C(T_i)$ of colors, so that every edge of the complement of $T_i$ is such that its incident vertices get different colors) such that $v_i$ is the only vertex with color $1$ in $T_i$. We then set $f$ as the indicator function
\begin{align*}
f_\mathrm{cc}(T(v)):=\left\{ \begin{array}{cl} 0, & \mbox{if } \mathcal{E}_v \mbox{ holds, } \\
1, & \mbox{otherwise.}\end{array}\right.
\end{align*}
We show now that the clique cover number of $T$ is equal to the number of vertices that were assigned $1$. Indeed, we will show that there exists an optimal clique coloring which uses that number of colors. This coloring will be constructed inductively over all layers bottom up. Moreover, we will simultaneously prove by induction that our coloring indeed is proper and optimal. The deepest layer contains the set of leaves. Every leaf is assigned $1$ under $f$ since there are no subtrees of the leaves.  Clearly all leaves are adjacent in the complement, so the set of leaves forms a clique in the complement, and thus all leaves must have different colors. The base case is satisfied.  Now, suppose inductively that for a layer $\ell$ with vertices $u_1, \ldots, u_{j_{\ell}}$, $\bigcup_{i=1}^{j_{\ell}} F_i$ is optimally colored (optimal in the sense of the clique cover number), where $F_i$ is the forest corresponding to the union of subtrees (at level $\ell-1$) pending from vertex $u_i$. Now, we color the $u_i$'s as follows: assume that for $F_i$ say $t$ colors are used. Shift these $t$ colors to the set $\{1,\ldots, t\}$ and then try all possible permutations of $\{1,\ldots,t\}$ to check whether there exists a permutation such that $\mathcal{E}_{u_i}$ holds. If there is a permutation such that $f(T(u_i))$ evaluates to $0$, assign to $u_i$ the same color before the shift of the colors that was used for the root that was assigned color $1$ after shifting and permuting colors. Otherwise, assign to $u_i$ a color which was not used yet. We have to show now that this coloring of $u_1,\ldots,u_{j_{\ell}}$ gives a proper and optimal coloring of $\bigcup_{i=1}^{j_{\ell}} T(u_i)$. First, we show that it is proper. Note that the $T(u_i)$'s are all colored properly by definition of the color of $u_i$ and the induction hypothesis which implies that any two vertices $k_1, k_2$ from different trees in $F_i$ have different colors. Moreover, again by induction hypothesis, any two vertices $k_1, k_2$ from different forests $F_i$ have also different colors. Thus, it suffices to show that all $u_i$'s are colored differently since the subgraph induced by these vertices form a clique in the complement. However, this is clear since the colors of the $u_i$'s either come from $F_i$ or are entirely new colors. Thus, the coloring is indeed proper. To show that the coloring is optimal, first note that the clique cover number is monotone under adding vertices: if two vertices need to be assigned different colors in a subtree (subforest), they still need to be assigned different colors after adding a new vertex. If a vertex $u_j$ is assigned $0$, then no new color is used for such a vertex, and this coloring remains optimal. If a vertex $u_j$ is assigned $1$, then note that $u_j$ must obtain a color different from all other vertices except for possibly those that are roots of the pending subtree (since $u_j$ is adjacent to all of them in the complement). If there were a coloring assigning $u_j$ the same color as the root of a pending subtree (and no other vertex of the subtree), then after permuting the colors one could assign to such a root color $1$, and to no other vertex in the subtrees of  $T(u_j)$ has color $1$, and hence $u_j$ would be assigned $0$, contradicting this possibility. Hence $u_j$ must be assigned a new color, and the coloring remains optimal.

Hence, the clique cover number of $T$ is given by $F_\mathrm{cc}(T)=\sum_{v\in T} f_\mathrm{cc}(T(v))$ as in (\ref{ref_F}). This implies that the clique cover numbers $C_n:=C({\cal T}_n) =F_\mathrm{cc}({\cal T}_n)$ and  $\widehat{C}_n:=C(\Lambda_n) =F_\mathrm{cc}(\Lambda_n)$ in distribution. We have $f_\mathrm{cc}(T(v))=\mathrm{O}(1).$ Hence, Corollary 1.15 of Holmgren and Janson~\cite{HJ15} implies the assertions of Theorems \ref{thm_bst} and \ref{thm_rrt}.\hfill$\square$\\

\noindent
{\bf Acknowledgement.} The results of the present note were obtained during the Twelfth Annual Workshop on Probability and Combinatorics at McGill University's Bellairs Research Institute. The authors thank the participants, in particular Luc Devroye and Remco van der Hofstad, for helpful discussions on the present problem. The hospitality and support of the institute is also acknowledged. When later also discussing our results with Stephan Wagner he told us that he too together with his PhD student Kenneth Dadedzi independently has shown results similar of our Theorems 1.1-1.2, see \cite{dadedzi}.

\end{document}